\documentclass[article, 11pt]{amsart}
\usepackage{amsfonts,amsthm,amsbsy,amsmath,amssymb,verbatim}
\usepackage{latexsym,enumerate,enumitem}
\usepackage{mathrsfs} 
\usepackage{tikz-cd}
\usetikzlibrary{positioning}
\usepackage{xcolor}
\usepackage{setspace}
\usepackage{geometry}
\usepackage{hyperref}
	\hypersetup{colorlinks, breaklinks,
            	linkcolor = blue,
				urlcolor = blue,
				anchorcolor = blue,
				citecolor = blue}

\usepackage{cleveref}

\newtheorem{lemma}{Lemma}

\newtheorem{theorem}{Theorem}

\newtheorem*{theorem*}{Theorem}
\newtheorem*{lemma*}{Lemma}
\newtheorem*{prop*}{Proposition}
\newtheorem*{corollary*}{Corollary}
\newtheorem*{remark*}{Remark} 
\newtheorem*{remarks*}{Remarks}
\newtheorem*{conj*}{Conjecture}

\def\N{{\mathbb N}}

\def\S{\mathbb{S}}

\def\R{{\mathbb R}}

\def\B{{\mathbb B}}

\def\e{\varepsilon}

\def\s{\sigma}

\DeclareFontFamily{U}{mathx}{\hyphenchar\font45}
\DeclareFontShape{U}{mathx}{m}{n}{
	<5> <6> <7> <8> <9> <10>
	<10.95> <12> <14.4> <17.28> <20.74> <24.88>
	mathx10
}{}

\begin{document}

 \title[The Multilinear Spherical Maximal Function in one dimension]{The Multilinear Spherical Maximal Function in one dimension}
\author[Georgios Dosidis and Jo\~{a}o P.G. Ramos]{Georgios Dosidis and Jo\~{a}o P.G. Ramos}

\newcommand{\Addresses}{{
		\bigskip
		\footnotesize
		
		\textsc{Georgios Dosidis.}
		\textsc{Department of Mathematics and Physics, Charles University,
			Prague, Czechia}\par\nopagebreak
		\textit{E-mail address:} \texttt{Dosidis@karlin.mff.cuni.cz}
		
		\medskip
		\textsc{Jo\~{a}o P.G. Ramos.}
		\textsc{Department of Mathematics, ETH Z\"urich,
			Z\"urich, Switzerland.}\par\nopagebreak
		\textit{E-mail address:} \texttt{Joao.ramos@math.ethz.ch}

}}
\thanks{The first author was supported by the Primus research programme PRIMUS/21/SCI/002 of Charles University. The second author was supported by the European Research Council under the Grant Agreement No. 721675 ``Regularity and Stability in Partial Differential Equations (RSPDE)''}

\begin{abstract} In dimension $n=1$ we obtain $L^{p_1}(\mathbb R) \times\dots\times L^{p_m}(\mathbb R)$ to $L^p(\mathbb R)$ boundedness for the multilinear spherical maximal function in the largest possible open set of indices  and we provide counterexamples  that  indicate the optimality of our results. 
\end{abstract}
\maketitle

\section*{Introduction}
The Lebesgue differentiation theorem states that, for a function $f \in L^1_{loc}(\R^d),$ there is a null set $E \subset \R^d$ so that, if $x \in \R^d \setminus E,$ then 
\[
\lim_{r \to 0} \frac{1}{|B_r(0)|} \int_{B_r(x)} f(y) \, dy = f(x).
\]
A natural question, in that regard, is whether the same convergence holds if one replaces averages over balls by averages over \emph{spheres}. In addition, the study of such spherical averages is deeply connected with the study of dimension-free bounds for the Hardy--Littlewood maximal function, as highlighted by Stein \cite{S1982}. 

In this direction, such a theorem on spherical averages induces the study of the \emph{spherical maximal function} defined by
\begin{equation}\label{SS}
	S (f) (x) := \sup_{t>0}\left| \int_{\mathbb S^{n-1}}f(x-t y) d\s_{n-1}(y)\right|.
\end{equation}
The study of bounds for the spherical  maximal function was initiated by Stein  \cite{S1976}, who obtained its boundedness from $L^p(\mathbb R^n) \to L^p(\mathbb R^n)$ when $n\geq 3$ and $p>\frac{n}{n-1}$ and showed that it is unbounded when $p\leq\frac{n}{n-1}$ and $n\geq 2$. The analogue of this result  in dimension $n=2$ was established later by Bourgain in \cite{B1986}, who also obtained a restricted weak type estimate in \cite{B1985} in the case $n\geq 3$. 

Further developments have been obtained by Seeger, Tao, and Wright, which, in \cite{STW2003}, proved that the restricted weak type estimate does not hold in dimension $n=2$. A number of other authors have also studied the spherical maximal function, among which we highlight  \cite{CM1979}, \cite{C1985}, \cite{R1986}, \cite{MSS1992}, \cite{S1998}, \cite{BOS2009} and the references therein. Extensions of the spherical maximal function to different settings have also been established by several authors; for instance, see \cite{C1979}, \cite{G1981}, \cite{DV1996}, and \cite{MSW2002}. 

The main object of this work is the $m$-linear analogue of the  spherical maximal function, given by
\begin{equation}\label{S^m}
	S^m (f_1,\dots,f_m) (x) := \sup_{t>0}\left| \int_{\mathbb S^{mn-1}}  \prod_{j=1}^{m} f_j(x-t y^j) d\s_{mn-1}(y^1,\dots,y^m)\right|,
\end{equation}
defined originally for Schwartz functions, where $d\s$ stands for the  (normalized) surface measure of $\mathbb S^{mn-1}$.

The $m=2$ case of \eqref{S^m} is called \emph{the bi(sub)linear}  spherical maximal function, and it was first introduced by Geba, Greenleaf, Iosevich, Palsson, and Sawyer \cite{GGIPS2013}, who obtained the first bounds for it.  Later improved bounds were provided  by \cite{BGHHO2018}, \cite{GHH2018}, \cite{HHY2019}, and \cite{JL2019}. A multilinear (non-maximal) version of this operator when all input functions lie in the same space $L^{p}(\R)$ was previously studied by Oberlin \cite{O1988}. 

It was not until the work of Jeong and Lee \cite{JL2019} that the sharp open range of boundedness would be proved for the bilinear operator. Indeed, the authors proved in \cite{JL2019} that when $n\geq 2$, the bilinear maximal function is pointwise bounded by the product of the linear spherical maximal function and the Hardy-Littlewood maximal function, which implies boundedness in the optimal open set of exponents. This was generalized to the multilinear setting in \cite{D2021}. See also \cite{AP20192}, \cite{AP20191}, \cite{DG2021}, and \cite{BFOPZ22} for further developments.

The purpose of this work is to complement the results of \cite{JL2019} and \cite{D2021} in the $n=1$ case. The spherical maximal operators are generally more singular when the dimension is smaller, which is reflected by the fact that the decay of the Fourier transform of the surface measure is smaller in low dimensions. 

Our first result establishes the $L^{p_1}\times L^{p_2}\to L^p$ boundedness of the one-dimensional bilinear operator ($m=2$) in the region $p_1,p_2>2$. This range is optimal due to a counterexample of Heo, Hong, and Yang \cite{HHY2019}, a variation of which also excludes the possibility of a weak-type bound when $p_1$ or $p_2$ equals $2$.

\begin{theorem}\label{th:S2}
	Let $p_1,p_2>1$ and $p=\frac{p_1p_2}{p_1+p_2}$. Then there is a constant $C=C(p_1,p_2)<\infty$ such that
	\begin{equation}\label{th:S2strongbd}\|S^2(f,g)\|_{L^p}\leq C \|f\|_{L^{p_1}}\|g\|_{L^{p_2}}
	\end{equation}
if and only if $p_1,p_2>2$. In this case $S^2$ admits a unique bounded extension from $L^{p_1}(\R)\times L^{p_2}(\R)$ to $L^{p}(\R)$.

Moreover, for the end cases $p_1=2$ and $p_2=2$ the bilinear spherical maximal function $S^2$ fails to be weak type bounded. In particular, for any $1\leq p_1,p_2\leq \infty$, $\frac{1}{p_1}+\frac{1}{p_2}=\frac{1}{p}$, $S^2$ does not boundedly map $L^2\times L^{p_2}\to L^{p,\infty}$ nor $L^{p_1}\times L^2\to L^{p,\infty}$.
\end{theorem}

\newcommand{\xx}{3.5}
\newcommand{\yy}{3.5}
\begin{figure}[h]
	\begin{center}
		\begin{tikzpicture}
			\coordinate (O) at (0,0);
			\coordinate (X) at (\xx,0);
			\coordinate (Y) at (0,\yy);
			\coordinate (T) at (\xx,\yy);
			\coordinate (HX) at (0.5*\xx,0);
			\coordinate (HY) at (0,0.5*\yy);
			\coordinate (C) at (0.5*\xx,0.5*\yy);
			\coordinate (s2) at (0.62*\xx,\yy);
			\coordinate (a1) at (\xx,0.83*\yy);
			\coordinate (a2) at (0.83*\xx,\yy);
			\coordinate (XE) at (1.2*\xx,0);
			\coordinate (YE) at (0,1.2*\yy);
			\draw[white,fill=yellow!.5] (O) -- (X) -- (T) -- (Y) -- (O) -- cycle;
			\draw[white,fill=yellow!18] (O) -- (HX) -- (C) -- (HY) -- (O) -- cycle;
			\draw[blue] (O) -- (HX);
			\draw[blue] (O) -- (HY);
			\draw[black](X) -- (T);
			\draw[black](Y) -- (T);
			\draw[black](HX) -- (X);
			\draw[black](HY) -- (Y);
			\draw[red!80] (HX) -- (C);
			\draw[red!80] (HY) -- (C);
			
			\foreach \rr in {HX,HY,C} {\filldraw[red](\rr) circle(1pt);}
			\foreach \rr in {O} {\filldraw[blue](\rr) circle(1pt);}
			\node [below = 1mm of X]  {(1,0)};
			\node [left  = 1mm of Y]  {(0,1)};
			\node [below = 1mm of HX]  {$(\frac12,0)$};
			\node [left =  1mm of HY]  {$(0,\frac12)$};
			\node [below left = 1mm of O] {$(0,0)$};
			
			\node [above  =1mm of XE]  {$1/p_1$};
			\node [above left =1mm of YE] {$1/p_2$};
			\draw[->,black,densely dotted] (X) -- (XE);
			\draw[->,black,densely dotted] (Y) -- (YE);

		\end{tikzpicture}
		\caption[Figure 1]{Range of $L^{p_1}\times L^{p_2}\to L^p$ boundedness of $S^2(f,g)$, when $n= 1$. }\label{F1}
	\end{center}
\end{figure}
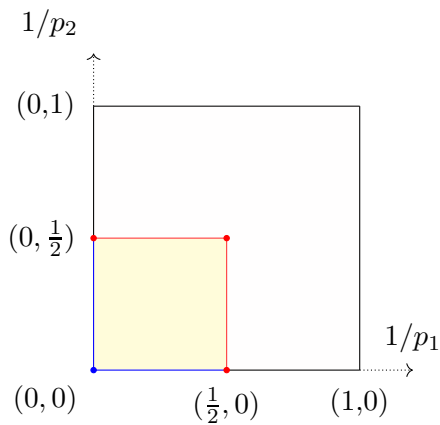

The proof of this result is based on a decomposition of the circle into sectors, in which we may safely parametrize it. We then use the curvature of the circle in our favour, in order to show a different kind of pointwise domination with respect to the $n \ge 2$ case: instead of bounding pointise by a product of the Hardy--Littlewood and spherical maximal functions, we obtain bounds with products of suitable $p-$maximal functions. In order to obtain these bounds, the curvature helps us by allowing us to insert power weights into the strategy, which effectively enable us to `transfer' decay from one maximal function to the other. 

Our second result deals with the multilinear case $m\geq 3$. Using the coarea formula (see \cite[Theorem 3.2.22]{federer}), we see that the following pointwise bound holds, for fixed $t>0$:
\begin{align*} & \left| S^m_t(f_1,\dots,f_m)(x) \right| = \left|  \int_{S^{m-1}} \prod_{k=1}^m f_k(x-ty_k) d\sigma(y) \right| \\
	&= \left| \int_{\B^{m-2}}\prod_{k=3}^m f_k(x-ty_k) \int_{r_y\S^1} f_1(x-ty_1)f_2(x-ty_2) d\s(y_1,y_2) \frac{dy_3\cdots dy_m}{r_y} \right| \\
	&\leq \left| \int_{\B^{m-2}}\prod_{k=3}^m f_k(x-ty_k) \int_{\S^1} f_1(x-tr_y y_1)f_2(x-tr_y y_2) d\s(y_1,y_2) dy_3\cdots dy_m \right| \\
	&\leq S^2(f_1,f_2)(x) \cdot M^{m-2}(f_3,\cdots,f_m)(x),
\end{align*}
where 
\[
M^m(f_1,\cdots,f_m)(x) = \sup_{t>0} \int_{\B^{m}}  \prod_{i=1}^m |f_1(x-ty_i)| dy_1\cdots d_m,
\]
is the $m$-linear Hardy-Littlewood maximal function, and $r_y=\sqrt{1-\sum_{k=3}^m y_k^2}$. From this estimate and interpolation, we obtain $L^{p_1}\times\cdots\times L^{p_m} \to L^p$ boundedness for $S^m$ in a certain range of exponents. The range of exponents thus obtained turns out to be the optimal boundedness range. 

\begin{theorem}\label{th:Sm}
	Let $n=1$, $m\geq 2$, $1\leq p_i\leq \infty$ for $i=1,\dots,m$, and $\displaystyle\frac{1}{p} = \sum_{i=1}^m\frac1{p_i}$. Then there is a constant $C<\infty$, only depending on $p_1,\dots,p_m$, such that
	\begin{equation}\label{eq:Smstrongbd}\|S^m(f_1,\dots,f_m)\|_{L^p(\R)}\leq C\prod_{i=1}^m \|f_i\|_{L^{p_i}(\R)}
	\end{equation}
	for all Schwartz functions $f_i$, $i=1,\dots,m$ if and only if all three of the following conditions hold:
	\begin{enumerate}[label=\alph*)]
		\item $\displaystyle\frac{1}{p} = \sum_{i=1}^m\frac1{p_i} < m-1$, 
		\item for every $i=1,\dots,m$, $\displaystyle{\sum_{j\neq i}\frac1{p_j}} < m-\frac32$,
		\item $\big(\frac1{p_1},\dots\frac{1}{p_m}\big) \not\in \{0,1\}^n\setminus\{(0,\dots,0)\}$.
	\end{enumerate}
	If  $\big(\frac1{p_1},\dots\frac{1}{p_m}\big) \in \{0,1\}^n\setminus\{(0,\dots,0)\}$, then we have the weak--type bound
	\begin{equation}\label{eq:Smweakbd}\|S^m(f_1,\dots,f_m)\|_{L^{p,\infty}(\R)}\leq C\prod_{i=1}^m \|f_i\|_{L^{p_i}(\R)}
	\end{equation}
	for some constant $C=C(p_1,\dots,p_m)$ if and only if (a) and (b) both hold.
	
	Finally, if for some $i=1,\dots,m$ we have $\displaystyle{\sum_{j\neq i}\frac1{p_j}} = m-\frac32$, then the bound in \cref{eq:Smstrongbd} cannot hold for any finite constant $C$.
\end{theorem}

As an example we graph the region of boundedness for the trilinear spherical maximal function.

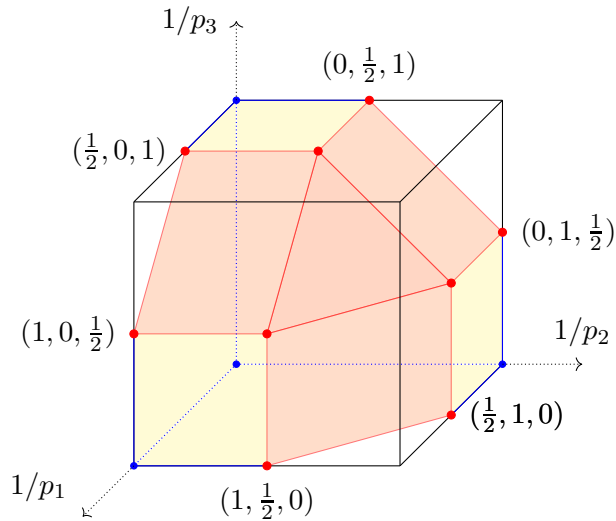
\begin{figure}[h]
	\newcommand{\Depth}{3.5}
	\newcommand{\Height}{3.5}
	\newcommand{\Width}{3.5}
	\begin{center}
		\begin{tikzpicture}
			\coordinate (O) at (0,0,0);
			\coordinate (A) at (0,0,\Height);
			\coordinate (B) at (\Depth,0,0);
			\coordinate (C) at (0,\Width,0);
			\coordinate (D) at (0,\Width,\Height);
			\coordinate (E) at (\Depth,0,\Height);
			\coordinate (F) at (\Depth,\Width,0);
			\coordinate (G) at (\Depth,\Width,\Height);
			\coordinate (hx) at (0,0,0.5*\Height);
			\coordinate (hy) at (0.5*\Depth,0,0);
			\coordinate (hz) at (0,0.5*\Width,0);
			\coordinate (hxY) at (\Depth,0,0.5*\Height);
			\coordinate (hyX) at (0.5*\Depth,0,\Height);
			\coordinate (hzY) at (\Depth,0.5*\Width,0);
			\coordinate (hxZ) at (0,\Width,0.5*\Height);
			\coordinate (hyZ) at (0.5*\Depth,\Width,0);
			\coordinate (hzX) at (0,0.5*\Width,\Height);
			\coordinate (hxhyZ) at (0.5*\Depth,\Width,0.5*\Height);	
			\coordinate (hyhzX) at (0.5*\Depth,0.5*\Width,\Height);	
			\coordinate (hzhxY) at (\Depth,0.5*\Width,0.5*\Height);	
			
			\coordinate (X) at (0,0,1.5*\Height);
			\coordinate (Y) at (1.3*\Depth,0,0);
			\coordinate (Z) at (0,1.3*\Width,0);

			\draw[white,fill=yellow!35] (C) -- (hxZ) -- (hxhyZ) -- (hyZ) -- cycle;
			\draw[white,fill=yellow!35] (A) -- (hyX) -- (hyhzX) -- (hzX) -- cycle;
			\draw[white,fill=yellow!35] (B) -- (hzY) -- (hzhxY) -- (hxY) -- cycle;
			\draw[white,fill=yellow!20] (O) -- (A) -- (hyX) -- (hxY) -- (B) -- cycle;
			\draw[white,fill=yellow!20] (O) -- (B) -- (hzY) -- (hyZ) -- (C) -- cycle;
			\draw[white,fill=yellow!20] (O) -- (A) -- (hzX) -- (hxZ) -- (C) -- cycle;
			\draw[red, fill=red!30,opacity=0.5] (hyhzX) -- (hzhxY) -- (hxhyZ) -- cycle; 
			\draw[red, fill=red!25,opacity=0.5] (hyhzX) -- (hzhxY) -- (hxY)-- (hyX) -- cycle; 
			\draw[red, fill=red!25,opacity=0.5] (hyhzX) -- (hxhyZ) -- (hxZ)-- (hzX) -- cycle; 
			\draw[red, fill=red!25,opacity=0.5] (hzhxY) -- (hxhyZ) -- (hyZ)-- (hzY) -- cycle; 
			
			
			\draw[black] (G) -- (D);	
			\draw[black] (G) -- (E);	
			\draw[black] (G) -- (F);	
			\draw[black] (A) -- (D);	
			\draw[black] (C) -- (D);	
			\draw[black] (F) -- (C);
			\draw[black] (F) -- (B);
			\draw[black] (E) -- (A);
			\draw[black] (E) -- (B);

			\draw[blue,densely dotted] (O) -- (A);
			\draw[blue,densely dotted] (O) -- (B);
			\draw[blue,densely dotted] (O) -- (C);
			\draw[blue] (A) -- (hyX);
			\draw[blue] (A) -- (hzX);
			\draw[blue] (B) -- (hxY);
			\draw[blue] (B) -- (hzY);
			\draw[blue] (C) -- (hxZ);
			\draw[blue] (C) -- (hyZ);


			\foreach \rr in {hxY, hxZ, hyX, hyZ, hzX, hzY, hxhyZ, hyhzX, hzhxY} {\filldraw[red](\rr) circle(1.5pt);}
			\foreach \rr in {O, A, B, C} {\filldraw[blue](\rr) circle(1.2pt);}
			\node [below =1mm of hyX]  {$(1,\frac12,0)$};
			\node [left = 1mm of hzX]  {$(1,0,\frac12)$};
			\node [right =1mm of hxY]  {$(\frac12,1,0)$};
			\node [right =1mm of hzY]  {$(0,1,\frac12)$};
			\node [above =1mm of hyZ]  {$(0,\frac12,1)$};
			\node [left =1mm of hxZ]  {$(\frac12,0,1)$};
			\node [right =1mm of hxY]  {$(\frac12,1,0)$};
			
			
			\node [above left =1mm of X]  {$1/p_1$};
			\node [above=1mm of Y] {$1/p_2$};
			\node [left=1mm of Z] {$1/p_3$};
			\draw[->,black,densely dotted] (A) -- (X);
			\draw[->,black,densely dotted] (B) -- (Y);
			\draw[->,black,densely dotted] (C) -- (Z);

		\end{tikzpicture}
		\caption[Figure 1]{The $L^{p_1}\times L^{p_2}\times L^{p_3}\to L^p$ boundedness region of the trilinear spherical maximal operator ($n = 1 $). }\label{F1}
	\end{center}
\end{figure}

In order to prove the necessity of the conditions on the exponents in Theorem \ref{th:Sm}, we shall employ two different kinds of counterexamples: the first is where all functions involved are similarly concentrated around the origin, which gives us condition (a), and the second in which all but one function - at entry $i$ - are similarly concentrated around the origin, whereas $f_i$ is spread out; this gives us condition (b). A modification of such examples in the spirit of Stein's original counterexample allows us to obtain condition (c) and the asserted lack of endpoint bounds.

Finally, let us mention for a brief moment the boundary case: for shortness of notation, define, for $i=1,\dots,m$, the sets $\mathcal H , \mathcal H_i$ as 
\[\mathcal H := [0,1]^m \cap \bigg\{ \sum_{j=1}^m x_j = m-1\bigg\} \bigcap\left[ \bigcap_{i=1}^m\Big\{\sum_{j\neq i} x_j \leq m-\frac32\Big\}\right] \]
and
\[\mathcal{H}_i = [0,1]^m \cap \bigg\{ \sum_{j\neq i} x_j = m-\frac32\bigg\}\bigcap \Big\{\sum_{j=1}^m x_j \leq m-1\Big\}.\]
In the diagram above, the set $\mathcal{H}$ denotes the middle triangle in red, whereas each of the $\mathcal{H}_i,i=1,2,3,$ denote one of the red rectangles. In spite of Theorem \ref{th:Sm} and the counterexamples it provides, the question of weak-type boundedness of $S^m$ when $(\frac1{p_1},\dots,\frac1{p_m})$ belongs in $\mathcal H$ or $\mathcal H_i$  remains open, as our counterexamples lie (sharply) in the corresponding Lebesgue spaces. 

\section*{Boundedness of the Multilinear Spherical Maximal function}

Let $m\in \N$ be the index of multilinearity, and $t>0$. Define for Schwartz functions $f_1,\dots,f_m$ on the real line 
\[
S^m_t(f_1,\dots,f_m)(x) =  \int_{\S^{m-1}} \prod_{i=1}^m f(x-ty_i) d\sigma(y),
\]
where $\S^{m-1}$ is the unit sphere in $\R^{m}$, $y=(y_1,\dots,y_m)\in \S^{mn-1}$, $y_i\in\R$ for $i=1,\dots,m$, and $d\sigma$ is the (normalized) surface measure on $\S^{m-1}$.
The multilinear spherical maximal operator is defined by
\[S^m(f_1,\dots,f_m)(x) = \sup_{t>0} S^m_t(|f_1|,\dots,|f_m|)(x) = \sup_{t>0} \int_{\S^{m-1}}  \prod_{i=1}^m |f_i(x-ty_i)| d\sigma(y).\]

\begin{proof}[Proof of \Cref{th:S2}, boundedness part]
By sublinearity, we can assume without a loss of generality that $f,g \geq 0$. Fix then two indices $p_1,p_2 > 2.$ 

 Decomposing the integral over $\S^1$ as the sum of the integrals over eight parts of the circle, we see that it is enough to deal with the integral over the set
\[\left\{(y_1,y_2)\in \S^1\,:\, 0\leq y_1 \leq \dfrac1{\sqrt{2}}\leq y_2\leq 1\right\},\]
as the treatment over the other sets is essentially equivalent. We then explicitly parametrize the circle over this arc, to obtain 
\begin{align*} S^2_{t}(f,g)(x) &= \int_{0}^{1/\sqrt{2}} f(x-ty_1)g(x-t\sqrt{1-y_1^2}) \frac{dy_1}{\sqrt{1-y_1^2}}\\
&\leq \int_{0}^{1/\sqrt{2}} f(x-ty_1)g(x-t\sqrt{1-y_1^2})dy_1\\
&=\int_{0}^{1/\sqrt{2}} f(x-ty_1)y_1^{-\frac{1-\varepsilon}2}g(x-t\sqrt{1-y_1^2}) y_1^{\frac{1-\varepsilon}2}dy_1\\
&\leq \left(\int_{0}^{1/\sqrt{2}} f^2(x-ty_1) y_1^{-1+\varepsilon} dy_1\right)^{1/2}\left( \int_{0}^{1/\sqrt{2}} g^2(x-t\sqrt{1-y_1^2}) y_1^{1-\varepsilon} dy_1	\right)^{1/2},
\end{align*}
where $\e>0$ small, to be chosen later. Since $y^{-1+\varepsilon}\chi_{0\leq y \leq {1/\sqrt{2}}}\in L^1$ and is decreasing for any $\varepsilon > 0,$ the maximal function
\[f\mapsto \sup_{t>0} \left(\int_{0}^{1/\sqrt{2}} f^2(x-ty_1) y_1^{-1+\varepsilon} dy_1\right)^{1/2}\]
is bounded on $L^{p_1}$, since $p_1>2$. For the second term, we change variables by setting $z=\sqrt{1-y_1^2}$ to get 
\begin{align*} &\left( \int_{0}^{1/\sqrt{2}} g^2(x-t\sqrt{1-y_1^2}) y_1^{1-\varepsilon} dy_1	\right)^{1/2}= \left( \int_{1/\sqrt{2}}^{1} g^2(x-tz) \left(\sqrt{1-z^2}\right)^{-\varepsilon} zdz	\right)^{1/2}\\
\leq &\left( \int_{1/\sqrt{2}}^{1} g^2(x-tz) \left(\sqrt{1-z^2}\right)^{-\varepsilon} dz	\right)^{1/2}\\
\leq & \left( \int_{1/\sqrt{2}}^{1} g^{2q}(x-tz) dz\right)^{1/2q}\left( \int_{1/\sqrt{2}}^{1} \dfrac1{\sqrt{1-z^2}^{\varepsilon q'} }dz	\right)^{1/2q'}
\end{align*}
for any $1\leq q,q'\leq \infty$ with $\frac{1}{q}+\frac1{q'}=1$. We choose $q$ sufficiently close to $1$ so that $2<2q<p_2$ and then we choose $\e$ to be sufficiently small so that $\e q'<2$. In this way the second term in the above product is finite, and the maximal function
\[g\mapsto \sup_{t>0} \left( \int_{1/\sqrt{2}}^{1} g^{2q}(x-tz) dz\right)^{1/2q}\]
is bounded on $L^{p_2}$. Finally, taking supremum over $t>0$ on both sides, we have
\begin{align*}  &S^2(f,g)(x) \\
&\lesssim \left(\sup_{t>0} \left(\int_{0}^{1/\sqrt{2}} f^2(x-ty_1) y_1^{-1+\varepsilon} dy_1\right)^{1/2}\right) \left(\sup_{t>0}\left( \int_{1/\sqrt{2}}^{1} g^{2q}(x-tz) dz\right)^{1/2q}\right).
\end{align*}
Taking $L^p$ norms on both sides and using H\"older's inequality and the bounds discussed above completes the proof of \cref{th:S2strongbd}.
\end{proof}

\begin{proof}[Proof of Theorem \ref{th:Sm}, boundedness part]

Again, by sublinearity, it is enough to assume that $f_i\geq 0$ for all $i=1,\dots,m$. Let $\left(\frac{1}{p_1},\dots,\frac{1}{p_m}\right)\in [0,1]^m$ be set of exponents such that $(a)$, $(b)$, and $(c)$ are satisfied. If there at least 2 indices $i_1$,$i_2$ such that $p_{i_1},p_{i_2}>2$ then the bound in \cref{eq:Smstrongbd} follows from the pointwise estimate 
\begin{equation}\label{eq:Smdec}
S^m(\vec{f})(x) \lesssim  S^2(f_{i_1},f_{i_2})(x) M^{m-2}(\pi_{i_1,i_2} (\vec{f}))(x) 
\end{equation}
where 
\[\pi_{i_1,i_2} \vec{f} (x) = (f_1,\dots,f_{i_1 -1},f_{i_1 +1},\dots,f_{i_2 -1},f_{i_2+1}\dots,f_m)(x) \]
is the projection of the vector $\vec{f}(x) $ from $\R^m$ to $\R^{m-2}$.
If $p_i\leq2$ for all $i=1,\dots,m$, then, by $(a)$, $\left(\frac{1}{p_1},\dots,\frac{1}{p_m}\right)$ belongs to the convex hull of $\mathcal H \cup \{0\}$. Similarly, if $p_i>2$ and $p_j\leq 2$ for all $j\neq i$, then $\left(\frac{1}{p_1},\dots,\frac{1}{p_m}\right)$ belongs in the convex hull of $\mathcal{H}_i\cup \{te_i, t \in (0,1/2)\}$. Using \Cref{th:S2} and multilinear interpolation (\cite{GLLZ2012}; see also \cite[Section 7.2]{GModern}), we obtain the strong and weak bounds in \cref{eq:Smstrongbd} and \cref{eq:Smweakbd}.
\end{proof}

\section*{Counterexamples}

\begin{proof}[Proof of Theorem \ref{th:S2}, counterexample part] 
In \cite{HHY2019} the authors showed that if the strong type bound \cref{th:S2strongbd} holds, then $p_1,p_2\geq 2$. Moreover, in \cite{BGHHO2018} it was shown that $p=\frac{p_1p_2}{p_1+p_2}>1$. A combination of the two examples shows that $p_1,p_2>2$ necessarily holds even for the weak type bound. We thus focus on this latter observation. 

Let $g=\chi_{[-10,10]}$ and $f(x)=|x|^{-1/2}\left(-\log(|x|)\right)^{-1}\chi_{[-1/2,1/2]}(x)$. Then $f\in L^2(\R)$ and $g\in L^{p_2}(\R)$ for any $p_2\geq 1$. For any $1/4\leq x\leq 1/2$ we choose $t=x$ in the definition of $S^2(f,g)$ to estimate it from below by
\begin{align*} S^2(f,g) (x) \geq& \int_{0}^{1} |x-xy|^{-1/2}\left(-\log(|x-xy|)\right)^{-1} \frac{dy}{\sqrt{1-y^2}}\\
	&\geq \frac{\sqrt{x}}{\sqrt{2}} \int_{0}^{1}(x-xy)^{-1} \left(-\log(x-xy)\right)^{-1}dy\\
	&\geq \frac{1}{\sqrt{2x}} \int_{0}^{x} u^{-1} \left(-\log(u)\right)^{-1}du=+\infty,
\end{align*}
where we changed variables $u = x - xy$ in the passage from the second to the third line. Therefore $S^2(f,g) (x) = +\infty$ on a set of positive measure and the result follows for the $p_1 =2$ case. Since the case $p_2=2$ is symmetric, this finishes our proof.
\end{proof}

\begin{proof}[Proof of Theorem \ref{th:Sm}, counterexample part] 
	We start by showing that the open set of exponents in \Cref{th:Sm} is optimal. 
	
\noindent\textit{Necessity of condition (a):} If $f_1=\cdots=f_m= \chi_{[-\delta,\delta]}$, then for $1/2\leq x\leq 1$ and $t=x\sqrt{m}$, a simple estimation of volume shows that
\[\hfill S^m(f_1,\dots,f_m)(x) \gtrsim \delta^{m-1} \qquad \text{for } \frac12\leq x\leq 1, \]
and thus, if $S^m$ is bounded from $L^{p_1} \times \cdots \times L^{p_m} \to L^p,$ we should have 
\[\delta^{m-1}\lesssim \|S^m(f_1,\dots,f_m)\|_{L^p}\lesssim \prod_{i=1}^m \|f_i\|_{L^{p_i}} \leq \delta^{\sum_{i=1}^m \frac{1}{p_i}},\]
and therefore $\sum\limits_{i=1}^m \frac{1}{p_i} \leq m - 1$. \\

\noindent\textit{Necessity of condition (b):} We set $f_1=\chi_{[-10\sqrt{m},10\sqrt{m}]}$ and $f_2=\cdots=f_m=\chi_{[-\delta,\delta]}$. For $1/2\leq x\leq 1$ we choose $t=x\sqrt{m-1}$ to estimate $S^m$ from below. This yields that 
\[\hfill S^m(f_1,\dots,f_m)(x) \gtrsim \delta^{m-3/2} \qquad \text{for } \frac12\leq x\leq 1, \]
by the tangency of this sphere to the $x_1$--axis. Thus, if $S^m$ is bounded for these exponents,
\[\delta^{m-\frac32}\lesssim \|S^m(f_1,\dots,f_m)\|_{L^p}\lesssim \prod_{i=1}^m \|f_i\|_{L^{p_i}} \lesssim_m \delta^{\sum_{i=2}^m \frac{1}{p_i}},\]
and therefore $\sum\limits_{i=2}^m \frac{1}{p_i} \leq m - \frac{3}{2}$.\\

With the necessity of conditions (a) and (b) in the statement of Theorem \ref{th:Sm} proved, we move on to proving that the strong-type bounds \emph{fail} in the boundary sets $\mathcal{H}$ and $\mathcal{H}_i, i=1,\dots,m.$ 

First of all, we note the following calculus fact: 

\begin{lemma} \label{lm:Calc} Let $r_1,r_2>0$, $t,s< e^{-\frac{r_2}{r_1}}$ and $t\leq C s$ for some $C\geq 1$. Then there exists an absolute constant $C'$ (depending only on $C,$ $r_1,$ $r_2$) such that
	\begin{equation}\label{l2} s^{-r_1}\left(\log \frac{1}{s}\right)^{-r_2}\leq C' t^{-r_1}\left(\log \frac{1}{t}\right)^{-r_2}.
	\end{equation}
\end{lemma}

 We then let $f_i= |x|^{-1/p_i}\left(-\log|x|\right)^{-2/p_i}\chi_{[-1/2,1/2]}$ for $i=1,\dots,m$, and note that $f_i\in L^{p_i}(\R)$. For large $x>0$, we choose $t=x\sqrt{m}$ to estimate $S^m(\vec{f})(x)$ from below by focusing on the region 
\[V_m(x):= \left\{(y_1,\dots,y_m)\in \S^{m-1} \, : \, \left|\frac1{\sqrt{m}} - y_1\right|, \cdots, \left|\frac1{\sqrt{m}} - y_{m-1}\right|< \frac1{300 m \cdot x\sqrt{m}}\right\}. \]
This yields the lower bound
\begin{align*} S^m(f_1,\dots,f_m) (x) &\geq \int_{V_m(x)} \prod_{i=1}^m f_i(x-\sqrt{m}xy_i)d\s(\vec{y})\\
	&\geq \int_{V_m(x)}\prod_{i=1}^m|x-\sqrt{m}xy_i|^{-1/p_i}\left(-\log(|x-\sqrt{m}xy_i|)\right)^{-2/p_i}d\s(\vec{y}). 
\end{align*}
Notice now that, for $\vec{y} \in V^{+}_m(x) = \{\vec{y}\in V_m(x)\, :\, y_m > 0 \},$ we have
\begin{align*}
\left| \frac{1}{\sqrt{m}} - y_m \right| & = \frac{1}{|\frac1{\sqrt{m}} + y_m|} \left|\frac{1}{m} - y_m^2 \right|  \leq \sqrt{m} \left| \frac{1}{m} - \left(1-\sum_{j \le {m-1}} y_j^2 \right) \right| \cr 
 & \leq  3 \sum_{j \le m-1} \left| y_j - \frac{1}{\sqrt{m}}\right|< 3(m-1) \frac{1}{300 m \cdot x\sqrt{m}}. \cr
&< \frac{1}{100 x  \sqrt{m}} .
\end{align*}
 This in turn implies that the new variables $u_i := x - x\sqrt{m} y_i, i=1,\dots,m,$ satisfy $(\sum_{i\le m-1} |u_i|^2)^{1/2}, |u_m| < e^{-2}$, which allows us to use \Cref{lm:Calc} since $\max_{i,j} \frac{p_i}{p_j} = 2$ for indices in $\mathcal H$.

With this in mind, we locally parametrize $V^{+}_m(x)$ in terms of the first $(m-1)$ coordinates and use the aforementioned change of variables $\vec{y} \mapsto \tilde{u}$ in the lower bound above, noticing we are in a position to use Lemma \ref{lm:Calc}, between $|u_i|$ and $|\tilde{u}|,$ where $\tilde{u} := (u_1,\dots,u_{m-1}).$ This implies, thus,
\begin{align*}
	S^m(f_1,\dots,f_m) (x) & \ge C_m |x|^{1-m} \int_{B^{m-1}(0,\frac1{300m})} |\tilde{u}|^{-\frac1p} \left(-\log(|\tilde{u}|)\right)^{-\frac{2}{p}}d\tilde{u}\\
	&\gtrsim \begin{cases} |x|^{1-m} & \text{if } \frac1p = m-1,\\
		\infty & \text{if } \frac1p >m-1. \end{cases}
\end{align*}
This deals with the lack of strong-type bounds for the set $\mathcal{H}.$ \\

We deal with the lack of strong-type bounds in each $\mathcal{H}_i$ in a similar manner. Without loss of generality we focus on $\mathcal{H}_m.$ Let then $f_i= |x|^{-1/p_i}\left(-\log|x|\right)^{-2/p_i}\chi_{[-1/2,1/2]}$ for $i=1,\dots,m-1,$ and $f_m = |x|^{-1/p_m} \left( \log |x|\right)^{-2/p_m} \chi_{\R \setminus [-2,2]}.$ Note that $f_i\in L^{p_i}(\R)$. For large $x>0$, we choose $t=x\sqrt{m-1}$ to estimate $S^m(\vec{f})(x)$ from below by focusing on the region 
\[W_m(x):= \left\{ \vec{y}\in \S^{m-1} \, : \, \left|\frac1{\sqrt{m-1}} - y_1\right|,\cdots, \left|\frac1{\sqrt{m-1}} - y_{m-1}\right|< \frac{10^{-4}}{ m x\sqrt{m-1}}\right\}, \]
over which $|1-y_m\sqrt{m-1}| \approx 1$. Moreover, it can be seen that - by similar methods to the ones employed in the analysis of $V_m(x)$ above - for $\vec{y} \in W_m(x),$ we have $\left( 1 - \sum_{i \le m-1} y_i^2 \right)^{-1/2} \ge c_m \left( \frac{x}{|\tilde{v}|} \right)^{1/2}, $
where $c_m >0$ is a constant depending only on $m,$ and $\tilde{v} = (v_1,\dots,v_{m-1}),$ where $v_i = x - x \sqrt{m-1}y_i.$ Parametrizing locally in terms of the first $(m-1)$ coordinates, changing variables $\vec{y} \mapsto \tilde{v}$ and using Lemma \ref{lm:Calc} again, we obtain 
\allowdisplaybreaks{
\begin{align*} S^m & (f_1,\dots,f_m) (x) \geq \int_{W_m(x)} \prod_{i=1}^{m} f_i(x-\sqrt{m-1}xy_i)d\s(\vec{y})\\
	\gtrsim &  x^{\frac32-m -\frac1{p_m}} \left(\log x\right)^{-\frac2{p_m}}\int_{B^{m-1}(0,\frac{10^{-4}}{m})} |\tilde{v}|^{- \frac12-\sum\limits_{i=1}^{m-1}\frac1{p_i} } \left(-\log(|\tilde{v}|)\right)^{-\sum\limits_{i=1}^{m-1}\frac2{p_i}}\, d\tilde{v}\\ 
	\gtrsim & \begin{cases} x^{-\frac1p} \left(\log x\right)^{-\frac2{p_m}} & \text{if } \sum\limits_{i\leq m-1}\frac1{p_i} = m-\frac32,\\
		\infty & \text{if } \sum\limits_{i\leq m-1}\frac1{p_i} >m-\frac32. \end{cases}
\end{align*}
}
Thus, when $\sum\limits_{i=1}^{m-1}\frac1{p_i} = m-\frac32$, the above calculation shows that $S^m(\vec{f})(x)\gtrsim  x^{-\frac1p} \left(\log\frac1x\right)^{-\frac2{p_m}}$ for $x$ sufficiently large, and thus $S^m(\vec{f})\not\in L^p$, since $\frac{2p}{p_m}<1$. This completes the proof of the fact that no strong-type bounds can hold in the sets $\mathcal{H}_i.$

Finally, suppose that (c) is not satisfied. The counterexample in \cite[Proposition 2]{D2021} shows that the strong-type bound in \cref{eq:Smstrongbd} cannot hold, since if, for instance, $p_1=\cdots=p_k=1$ and $p_{k+1}=\cdots=p_m=\infty$, we may take $f_1=\cdots=f_k=\chi_{(-1,1)}$ and $f_{k+1}=\cdots=f_m \equiv 1$. Then for large $x>0$ and $t=x\sqrt{k}$
\begin{align*}
S^m(f_1,\dots,f_m)(x)  &\geq \int_{B^{k}(0,1)} \prod_{i=1}^k |f_i(x-x\sqrt{k}y_i)| dy_1\dots dy_k\\
&\gtrsim |x|^{-k}
\end{align*}
pointwise, which shows that \cref{eq:Smstrongbd} cannot hold in this case.
\end{proof}

\bibliographystyle{amsplain}

\Addresses

\end{document}